\documentclass[preprint,12pt,authoryear]{elsarticle}

\usepackage{amssymb}
\usepackage{amsmath}
\usepackage{amsthm}

\usepackage{geometry} 
\usepackage{hyperref}
\hypersetup{
    colorlinks=true,
    linkcolor=blue,
    filecolor=magenta,      
    urlcolor=cyan,
}
\usepackage{subfig} 

\journal{arXiv}
\date{}

\newtheorem{lemma}{Lemma}
\newtheorem{proposition}{Proposition}
\newtheorem{Result}{Result}

\usepackage{mathtools}
\DeclarePairedDelimiter\ceil{\lceil}{\rceil}
\DeclarePairedDelimiter\floor{\lfloor}{\rfloor}
\usepackage{enumerate}  

\newenvironment{tight_enumerate}{
\begin{enumerate}[(i)]
  \setlength{\itemsep}{0pt}
  \setlength{\parskip}{0pt}
}{\end{enumerate}} 

\begin{document}

\begin{frontmatter}



\title{On arbitrarily underdispersed Conway--Maxwell--Poisson distributions}


\author{Alan Huang}

\address{School of Mathematics and Physics, University of Queensland, QLD, Australia}

\begin{abstract}
We show that the Conway--Maxwell--Poisson distribution can be arbitrarily underdispersed when parametrized via its mean. 
More precisely, if the mean $\mu$ is an integer then the limiting distribution is a unit probability mass at $\mu$. If the mean $\mu$ is not an integer then the limiting distribution is a shifted Bernoulli on the two values $\floor{\mu}$ and $\ceil{\mu}$ with probabilities 
equal to the fractional parts of $\mu$. In either case, the limiting distribution is the most underdispersed discrete distribution possible for any given mean. 
This is currently the only known generalization of the Poisson distribution 
exhibiting this property. 
Four practical implications are discussed, each adding to the claim that the (mean-parametrized) Conway--Maxwell--Poisson distribution should be considered the default model for underdispersed counts. We suggest that all future generalizations of the Poisson distribution be tested against this property.

\end{abstract}




\end{frontmatter}


\section{Introduction}

The Conway--Maxwell--Poisson (CMP) distribution
is a generalization of the Poisson distribution that 
has seen a recent revival in popularity for the modelling of both underdispersed and overdispersed counts \citep[see, e.g.,][]{SMKBB:2005, SS:2010, LGD:2010, FGD:2019, SP:2020}. The probability mass function (pmf) of the CMP distribution is given by
\begin{equation}
P(Y = y) = \frac{1}{Z(\lambda, \nu)} \frac{\lambda^y}{(y!)^\nu} \ , \quad y=0,1,2,\ldots,
\label{eq:pmf}
\end{equation}
where $\lambda \ge 0$ is a rate parameter, $\nu \ge 0$ is a dispersion parameter, and $Z(\lambda, \nu) = \sum_{y=0}^\infty \lambda^y/(y!)^\nu $ is a normalizing function. The CMP distribution can also be characterized via its ratio of successive probabilities,
\begin{equation}
\frac{P(Y = y-1)}{P(Y = y)} = \frac{y^\nu}{\lambda} \ , \quad y = 1,2,3,\ldots \ .
\label{eq:successive}
\end{equation}

A key feature of the CMP distribution is that it forms a continuous bridge between some well-known distributions, passing through the overdispersed geometric($\lambda$) distribution when $\nu = 0$, the equidispersed Poisson$(\lambda)$ distribution when $\nu = 1$, and the underdispersed Bernoulli$(\lambda/(1+\lambda))$ distribution as $\nu \to\infty$ \citep[][p. 129]{SMKBB:2005}. These results are for fixed rate $\lambda$.

This note generalizes the last result by allowing the rate $\lambda = \lambda(\mu,\nu)$ to vary with $\nu$ in such a way that the mean $\mu$ of the distribution remains fixed. Under this mean parametrization, we have the following asymptotic behaviour for arbitrarily small underdispersion.

\begin{proposition}
\label{eq:prop1}
As $\nu \to \infty$, the CMP distribution with mean $\mu \ge 0$ converges to 
\begin{tight_enumerate}
\item a unit point mass at $\mu$ if $\mu$ is integer , i.e., $P(Y = \mu) \to 1$.
\item a shifted Bernoulli on the two integers $\floor{\mu}$ and $\ceil{\mu}$  if $\mu$ is non-integer, with probabilities equal to the fractional parts of $\mu$, i.e., $P(Y = \floor{\mu}) \to 1-\Delta$ and $P(Y = \ceil{\mu})\to \Delta$, where $\Delta = \mu - \floor{\mu}$.
\end{tight_enumerate}
\end{proposition}

The convergence paths of the two limiting cases are visualised in Figure \ref{fig:pmfs}. We see, for example, that an increasingly underdispersed CMP distribution with mean $\mu=4$ converges to a single probability mass at 4, while an increasingly underdispersed CMP distribution with mean $\mu=4.321$ converges to a shifted Bernoulli on values 4 and 5 with probabilities 0.679 and 0.321, respectively. In either case the limiting distribution is the most underdispersed discrete distribution possible for a given mean. To the best of our knowledge, this is currently the only known generalization of the Poisson distribution that exhibits this property.

\begin{figure}
\includegraphics[width = \textwidth]{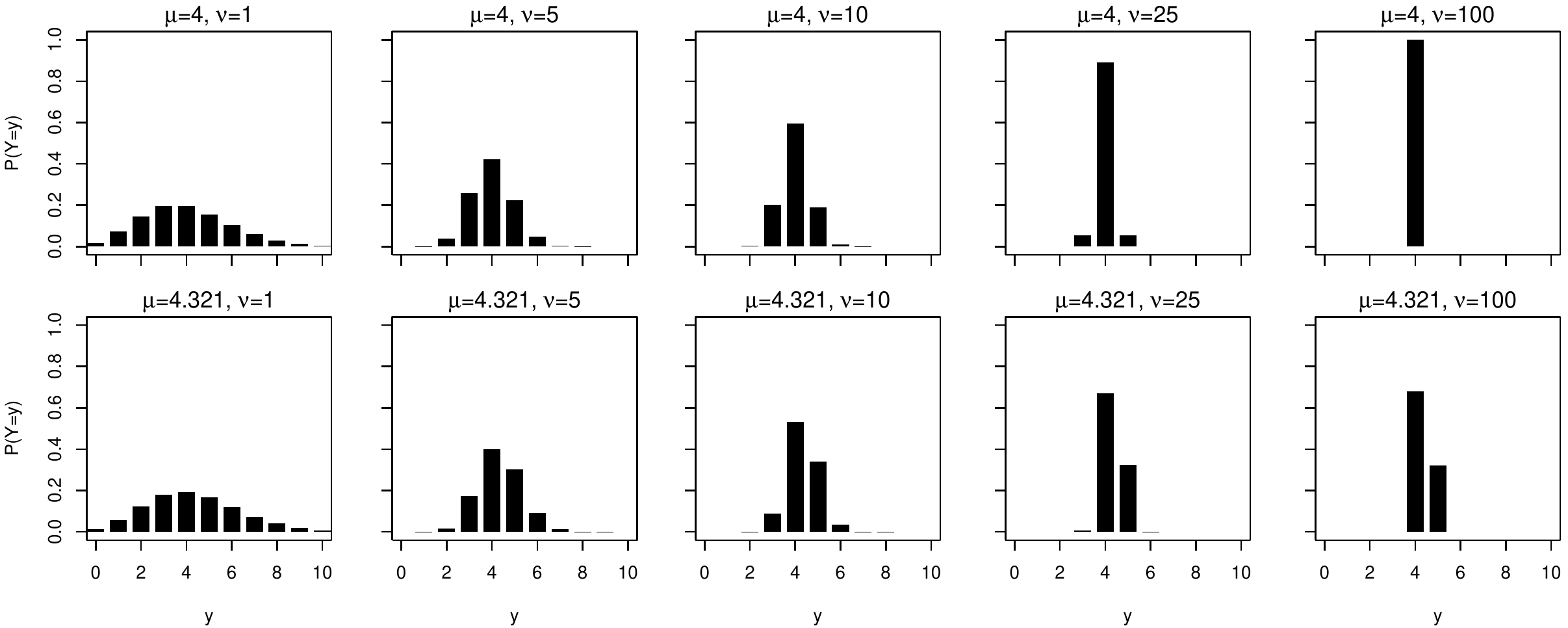}
\caption{pmfs of CMP distributions with means 4 and 4.321, and with dispersion increasing from $\nu = 1$ (Poisson) to $\nu=5, 10, 25$ and $100$ (severely underdispersed).}
\label{fig:pmfs}
\end{figure}

The proof of Proposition~\ref{eq:prop1} is provided in the Appendix and makes use of the following two lemmas which provide bounds on the rate $\lambda = \lambda(\mu, \nu)$ as $\nu$ gets arbitrarily large. The proofs of Lemmas~\ref{lem:1} and~\ref{lem:2} are also provided in the Appendix.

\begin{lemma}
\label{lem:1}
For any $\mu > 0$, the solution $\lambda = \lambda(\mu, \nu)$ to the mean constraint (\ref{eq:mconstraint}) satisfies $\lambda/\mu^\nu \to \infty$ and $\lambda/ (\mu+1)^\nu \to 0$ as $\nu \to \infty$.
\end{lemma}


\begin{lemma}
\label{lem:2}
If $\mu$ is non-integer, the bounds on the solution $\lambda = \lambda(\mu, \nu)$ to the mean constraint (\ref{eq:mconstraint}) can be tightened to $\lambda > \Delta \ceil{\mu}^\nu$ and $\lambda < (1-\Delta)^{-1} \ceil{\mu}^\nu$ for sufficiently large $\nu$.
\end{lemma}

\section{Practical implications}
\subsection{Functional independence of parameters $\mu$ and $\nu$}
It is already known that the mean $\mu \ge 0$ and dispersion $\nu \ge 0$ in the mean-parametrized CMP distribution  are orthogonal \citep[][Result 2]{Huang:2017}. Proposition~\ref{eq:prop1} demonstrates that the two parameters are also functionally independent. This makes it unique amongst other generalizations of the Poisson distribution, including the generalized Poisson \citep{Consul:1989} as implemented in the \texttt{VGAM} package \citep{Yee:2020}, hyper-Poisson \citep{SC:2013}, extended Poisson-Tweedie \citep{BJKHD:2018}, Bernoulli-geometric \citep[BerG, as implemented in][]{Matheus:2020}, and the original CMP \citep{SMKBB:2005} as implemented in \texttt{COMPoissonReg} \citep{SS:2019}, all of which place restrictions on one  parameter based on the value(s) of the other parameter(s) -- see column 2 of Table~\ref{tab:1}. The mean-parametrized CMP is therefore the only model that can be fit to any count dataset regardless of the combination of mean and dispersion exhibited by the data.


Consider a simple set of counts $ \boldsymbol{y} = (26, 27, 27, 28, 28, 28, 28)^\top$ with sample mean 27.43 and variance 0.62, which is severely underdispersed for discrete data. The ``perfect" model fit here is given by the empirical distribution with $\hat p_{26} = 1/7, \hat p_{27} = 2/7$ and $\hat p_{28} = 4/7$, which attains the highest possible log-likelihood of $-3.758$ and lowest possible AIC of 11.52. For comparisons, the maximum likelihood estimates for each of the above models, along with their fitted means, variances and AICs, are given in Table~\ref{tab:1}. 

\begin{table}
\begin{center}
\scriptsize
\begin{tabular}{llllc}
\hline
Model & restrictions & MLE &  fitted values & AIC \\
\hline
generalized Poisson$^{[1]}$ & $\max(-1, -\theta/m) \le \lambda \le 1,$ &
$\hat \lambda = -0.999,$ & $\hat \mu = 28.92$  & 33.53 \\
\quad $\lambda \in (-1, 1), \theta > 0$ & where $m$ is the largest integer & $\hat \theta = \ 57.834$ & $\hat \sigma^2 =  7.23$ \\
& satisfying $\theta + m\lambda >0$ if $\lambda < 0$ & \\
hyper-Poisson$^{[1]}$ & $\lambda \ge \min\{\mu, \max(\mu + (\gamma -1), \gamma \mu \}$ & $\hat \lambda = 26.43$ & $\hat \mu = 27.43$ & 39.97 \\ 
\quad $\lambda > 0, \gamma > 0$ & \& $\lambda \le \max\{\mu, \min(\mu+(\gamma-1), \gamma \mu) \}$ & $\hat \gamma = 3.96\times 10^{-13}$ & $\hat \sigma^2 = 26.43 $\\
BerG & $\phi > |\mu-1|$ & $\hat \mu = 27.43,$ & $\hat \mu =27.43$  & 64.10 \\
\quad $\mu > 0$  $\phi > 0$ & & $\hat \phi = 26.43$  & $\hat \sigma^2 =724.90$ \\
Poisson-Tweedie$^{[2]}$ & $\phi > - \mu^{(1-p)}$ & \multicolumn{1}{l}{does not exist} & --- & --- \\
\quad $\mu > 0, p \ge 0,  \phi$ & \\
CMP$^{[1]}$ & $\lambda > [(\nu-1)/2\nu]^\nu$ & \multicolumn{1}{l}{did not converge} & --- & ---\\
\quad $\lambda \ge 0 , \nu \ge 0$ \\
mean-parametrized CMP & none & $\hat \mu = 27.43,$ & $\hat \mu = 27.43$  & 19.48 \\
\quad $\mu \ge 0, \nu \ge 0$ & &  $\hat \nu = 52.26$ & $\hat \sigma^2= \ 0.53$ \\
\hline
\end{tabular}
\end{center}
\vspace{-5mm}
\caption{Parameter spaces, restrictions and estimates, along with fitted means, variances and AIC values, of competing count distributions applied to the dataset $ \boldsymbol{y} = (26, 27, 27, 28, 28, 28, 28)^\top$ with sample mean $\bar{\boldsymbol{y}} = 27.43$, sample variance ${\rm var}(\boldsymbol{y}) = 0.62$, and best possible AIC of 11.52. Notes: ${[1]}$ susceptible to convergence issues at the boundary of parameter space; ${[2]}$ extended Poisson-Tweedie pmf does not exist when underdispersed ($\phi \le 0$).}
\label{tab:1}
\end{table}

We see that while all models fit the mean of the data well, only the mean-parametrized CMP can simultaneously adapt to the severe underdispersion exhibited by the data, attaining an AIC that is closest to the lowest possible value. This is because the strong functional dependence of parameters in the other models restricts the level of underdispersion allowed -- in particular, the larger the mean count the less underdispersion is permissible. For example, the most underdispersed hyper-Poisson distribution is obtained by taking $\gamma \to 0$ which implies that the smallest possible variance is $\mu - 1$ for any mean $\mu > 1$.  Thus, for a mean of 27.43 the most underdispersed hyper-Poisson distribution has variance 26.43. None of the other distributions fare any better: the underdispersed generalized Poisson pmf does not necessarily sum to 1, the BerG distribution simply cannot be underdispersed if mean $\mu \ge 2$, and the underdispersed extended Poisson–Tweedie pmf does not even exist (!).

\subsection{Generating underdispersed counts}

The last point above also implies that the mean-parametrized CMP distribution is the only candidate amongst these models that remains a full probability model on the non-negative integers over its entire parameter space. 
It is therefore the only candidate that can also be used to
%
generate arbitrarily underdispersed counts, which is particularly useful for simulation studies as in \cite{FGD:2019}.

\subsection{Improved computation speed for evaluating the CMP distribution}
The bounds given in Lemmas~\ref{lem:1} and~\ref{lem:2} provide a tight range in which to numerically search for $\lambda$ given the mean and dispersion. This is especially useful in practice because the problem of solving for both the rate $\lambda$ and the normalizing function $Z$ becomes computationally demanding with increasingly small underdispersion (i.e., large $\nu$).  Practically speaking we have found that these bounds already hold for $\mu \ge 1$ and $\nu \ge 1$, and using them reduces computation time by at least one order of magnitude from the original implementation in the \texttt{mpcmp} package \citep{FAWH:2020}. These bounds also allow for linear interpolation in $\nu$ and log-linear interpolation in $\mu$ when $\log \lambda$ is evaluated on a grid, allowing for fast, approximate updates for Bayesian calculations. 

\subsection{Second-order consistent discrete kernel smoothing}
\cite{KK:2011} defined the concept of a second-order discrete associated kernel function $f_{xh}(\cdot)$ as a discrete analogue to continuous kernel functions satisfying
$$
\lim_{h \to 0} E(f_{xh}) = x \quad \text{ and } \quad  \lim_{h \to 0} {\rm Var}(f_{xh}) = 0
$$
for every integer $x \in \mathbb{N}$, where $h \ge 0$ is a bandwidth parameter that acts like the variance in a Gaussian kernel smoother. The second condition here is precisely the requirement that the class of distributions $f_{xh}(\cdot)$ can be arbitrarily underdispersed, which is needed for constructing consistent discrete kernel smoothers. Proposition~\ref{eq:prop1} implies that the mean-parametrized CMP is one such example of a second-order discrete associated kernel, making it a natural candidate for constructing consistent discrete kernel smoothers for count data \citep[see][]{HSF:2020}. In fact, the mean-parametrized CMP distribution is currently the the only non-trivial discrete distribution satisfying these requirements -- the other two examples 
in the literature 
being the ``trivial" (unsmoothed) histogram and triangular kernel smoother of \cite{KKZ:2007}. 

\section{Conclusion}
The CMP distribution can handle arbitrarily small underdispersion when parametrized via its mean, with the limiting distribution (either a single probability mass or a shifted Bernoulli) being the most underdispersed possible for any discrete distribution. It is currently the only known generalization of the Poisson distribution possessing this property. 
The practical implications of this result add to the increasingly strong case for the CMP distribution to be the default model for underdispersed counts. Thus, we propose that all generalizations of the Poisson distribution be tested against this property. 

Future research into the rates of convergence in Proposition~\ref{eq:prop1}, as well as the behaviour of {\it sums} of independent mean-parametrized CMP random variables (which form a continuous bridge between the overdispersed negative-binomial, equidispersed Poisson and (arbitrarily) underdispersed shifted Binomial or single point mass distributions), are also warranted.

\section*{Acknowledgements} The author thanks Lucas Sippel (UQ) and Thomas Fung (Macquarie) for discussions leading to the writing of this paper, and Thomas Yee (Auckland) for insightful comments that much improved the paper.

\newpage
\appendix
\normalsize

\section{Proof of Lemmas 1 and 2}
\label{app:A}
For the rate $\lambda = \lambda(\mu, \nu)$ to vary with the dispersion $\nu$ such that the mean $\mu$ remains fixed, it must satisfy the mean constraint,
\begin{equation}
\label{eq:mconstraint}
0 = \sum_{y=0}^\infty (y -\mu) \frac{\lambda^y}{(y!)^\nu} \ .
\end{equation}
Note that setting $\lambda = \lambda(\mu, \nu)$ in the pmf (\ref{eq:pmf}) leads to the mean-parametrized CMP distribution of \cite{Huang:2017}. The following result is then used  to establish the bounds on $\lambda(\mu, \nu)$ given by Lemma~\ref{lem:1} for the case of integer $\mu$; the result for non-integer $\mu$ is covered by Lemma~~\ref{lem:2} which is proven later. The proof of Result~\ref{res:1} is given in \ref{app:B} of this supplement.

\begin{Result}
Fix an integer $\mu \ge 1$. Then the function $\mu^y/y!$ is
\begin{tight_enumerate}
\item  strictly increasing from $y=0$ to $\mu-1$ ;
\item strictly decreasing from $y = \mu$ to $\infty$ ;
\item strictly less than 1 for $y \ge 2 \mu^2$.
\end{tight_enumerate} 
\label{res:1}
\end{Result}

\begin{proof}[Proof of Lemma~\ref{lem:1} for integer $\mu$]
First, note that for each $\nu$ the CMP is a linear exponential family with canonical parameter $\log \lambda$ \citep[see][Section 3.2]{SMKBB:2005}. By properties of exponential families, the mean is a monotonic function of the canonical parameter and therefore a monotonic function of $\lambda$ also. Thus, the mean constraint~(\ref{eq:mconstraint}) has (at most) one solution $\lambda = \lambda(\mu, \nu)$ for each $\mu$ and $\nu$. 

Next, write $m(\lambda) = \sum_{y=0}^\infty (y-\mu) \lambda^y/(y!)^\nu$ so that the solution $\lambda = \lambda(\mu, \nu)$ is the root of $m$.  Consider evaluating $m$ at $\lambda = a \mu^\nu$ for some fixed $a > 0$. Writing out the summation in $m$ explicitly into three parts, one corresponding to $y \le \mu-1$, another for $ \mu+1 \le y \le 2\mu^2-1$, and one for $y \ge 2\mu^2$, gives
\begin{eqnarray*}
m(a \mu^\nu) &=& \underbrace{\sum_{y=0}^{\mu-1} (y - \mu) a^y \left[\frac{\mu^y}{(y!)}\right]^\nu}_{\text{negative part}} + \underbrace{\sum_{y=\mu+1}^{2\mu^2-1} (y - \mu) a^y \left[\frac{\mu^y}{(y!)}\right]^\nu}_{\text{positive part}} + \underbrace{\sum_{y=2\mu^2}^{\infty} (y - \mu) a^y \left[\frac{\mu^y}{(y!)}\right]^\nu}_{\text{remainder}} \ .
\end{eqnarray*}
As $\nu \to \infty$ each term in the remainder sum tends to 0 by Result~\ref{res:1}(iii). Moreover, using Result~\ref{res:1}(ii) and the monotone convergence theorem, the remainder sum tends to 0 and so is negligible for arbitrarily large $\nu$. 

By the strictly increasing property in Result~\ref{res:1}(i), for arbitrarily large $\nu$ the negative part is dominated by its last term $y=\mu-1$ in the sum, which is 
$$
- a^{\mu-1} \left[\frac{\mu^{\mu-1}}{(\mu-1)!}\right]^\nu \ .
$$
Similarly, by the strictly decreasing property in Result~\ref{res:1}(ii), for arbitrarily large $\nu$ the positive part is dominated by its first term $y = \mu+1$ in the sum, which is
$$
+ a^{\mu+1} \left[\frac{\mu^{\mu+1}}{(\mu+1)!}\right]^\nu \ .
$$

Thus for arbitrarily large $\nu$ the sign of $m(a \mu^\nu)$ is determined by the sign of the sum of the two dominant terms,
$$
a^{\mu+1} \left[\frac{\mu^{\mu+1}}{(\mu+1)!}\right]^\nu - a^{\mu-1} \left[\frac{\mu^{\mu-1}}{(\mu-1)!}\right]^\nu \ .
$$
By considering the ratio of these two terms, 
$$
\frac{a^{\mu+1} \left[\frac{\mu^{\mu+1}}{(\mu+1)!}\right]^\nu}{a^{\mu-1} \left[\frac{\mu^{\mu-1}}{(\mu-1)!}\right]^\nu} = a^2 \left[ \frac{\mu}{\mu+1} \right]^\nu,
$$
we see that for any fixed $a > 0$, $\nu$ can be chosen sufficiently large so that this ratio is less than 1. Conclude that for any $a > 0$, $m(a \mu^\nu)$ is negative for sufficiently large $\nu$.

Now consider $\lambda = b (\mu+1)^\nu$ for some fixed $b>0$. By analogous arguments, the sign of $m(b (\mu+1)^\nu)$ is determined by the sign of the sum of the two dominant terms 
$$
b^{\mu+1} \left[\frac{(\mu+1)^{\mu+1}}{(\mu+1)!}\right]^\nu - b^{\mu-1} \left[\frac{(\mu+1)^{\mu-1}}{(\mu-1)!}\right]^\nu \ .
$$
Considering the ratio of these two terms, 
$$
\frac{b^{\mu+1} \left[\frac{(\mu+1)^{\mu+1}}{(\mu+1)!}\right]^\nu}{b^{\mu-1} \left[\frac{(\mu+1)^{\mu-1}}{(\mu-1)!}\right]^\nu} = b^2 \left[ \frac{\mu+1}{\mu} \right]^\nu \ ,
$$
we see that for any fixed $b > 0$, $\nu$ can be chosen sufficiently large so that this ratio is larger than 1. Conclude that for any $b > 0$, $m(b(\mu+1)^\nu)$ is positive for sufficiently large $\nu$. 

Thus, for any $a, b > 0$ the solution $\lambda = \lambda(\mu, \nu)$ of mean constraint $m(\lambda) = 0$ is bounded between $a\mu^\nu$ and $b (\mu+1)^\nu$ for sufficiently large $\nu$. Hence, it must be that $\lambda/\mu^\nu \to \infty$ and $\lambda/(\mu+1)^\nu \to 0$, which proves Lemma~\ref{lem:1} for integer $\mu$. The result for non-integer $\mu$ is covered by Lemma~\ref{lem:2} below.
\end{proof}

To show Lemma~\ref{lem:2}, we use the following analogue to Result~\ref{res:1}. The proof of Result~\ref{res:2} is essentially identical to Result~\ref{res:1} and is omitted.

\begin{Result}
Fix a non-integer $\mu > 0$. Then the function $\mu^y/y!$ is
\begin{tight_enumerate}
\item strictly increasing from $y=0$ to $\floor{\mu}$
\item strictly decreasing from $y=\ceil{\mu}$ to $\infty$
\item strictly less than 1 for $y \ge 2 \ceil{\mu}^2$
\end{tight_enumerate}
\label{res:2}
\end{Result}

\begin{proof}[Proof of Lemma~\ref{lem:2}] Consider evaluating $m$ at $\lambda = \Delta \ceil{\mu}^\nu$, where $\Delta = \mu - \floor{\mu}$ and $1-\Delta = \ceil{\mu} - \mu$ are the fractional parts of $\mu$. By the same steps as in the proof of Lemma~\ref{lem:1}, we can write $m( \Delta \ceil{\mu}^\nu)$ as a sum of its negative part ($y \le \floor{\mu}$), positive part ($\floor{\mu} \le y \le 2\ceil{\mu}^2-1$), and its remainder part ($y \ge 2\ceil{\mu}^2$). Using Result~\ref{res:2}, as in the proof of Lemma~\ref{lem:1}, the remainder component tends to 0 for large $\nu$, the negative part is dominated by its last term ($y = \floor{\mu}$), and the positive part of the sum is dominated by its first term ($y=\ceil{\mu}$). Thus, the sign of $m( \Delta \ceil{\mu}^\nu)$ is determined by the sum of its dominant positive and negative terms,
$$
(1-\Delta) \Delta^{\ceil{\mu}} \left[\frac{\ceil{\mu}^{\ceil{\mu}}}{\ceil{\mu}!}\right]^\nu - \Delta \Delta^{\floor{\mu}} \left[\frac{\ceil{\mu}^{\floor{\mu}}}{\floor{\mu}!}\right]^\nu \ .
$$
Evaluating the the ratio of these two terms gives
$$
\frac{(1-\Delta) \Delta^{\ceil{\mu}} \left[\frac{\ceil{\mu}^{\ceil{\mu}}}{\ceil{\mu}!}\right]^\nu}{\Delta \Delta^{\floor{\mu}} \left[\frac{\ceil{\mu}^{\floor{\mu}}}{\floor{\mu}!}\right]^\nu} = (1-\Delta) < 1 .
$$
Conclude that $m( \Delta \ceil{\mu}^\nu)$ is negative for sufficiently large $\nu$.

Now consider evaluating $m$ at $\lambda = (1-\Delta)^{-1} \ceil{\mu}^\nu$. By the same arguments as in before, the sign of $m((1-\Delta)^{-1} \ceil{\mu}^\nu)$ is determined by the sign of the sum of its dominant negative and positive terms,
$$
(1-\Delta) (1-\Delta)^{-\ceil{\mu}} \left[\frac{\ceil{\mu}^{\ceil{\mu}}}{\ceil{\mu}!}\right]^\nu - \Delta (1-\Delta)^{-\floor{\mu}} \left[\frac{\ceil{\mu}^{\floor{\mu}}}{\floor{\mu}!}\right]^\nu \ .
$$
Evaluating the the ratio of these two terms gives
$$
\frac{(1-\Delta) (1-\Delta)^{\floor{\mu}} \left[\frac{\mu^{\ceil{\mu}}}{\ceil{\mu}!}\right]^\nu}{\Delta (1-\Delta)^{\ceil{\mu}} \left[\frac{\mu^{\floor{\mu}}}{\floor{\mu}!}\right]^\nu} = \frac{1}{\Delta} > 1 \ .
$$
Conclude that $m((1- \Delta)^{-1} \ceil{\mu}^\nu)$ is positive for sufficiently large $\nu$. Hence for non-integer $\mu$ the solution $\lambda = \lambda(\mu, \nu)$ of mean constraint $m(\lambda) = 0$ must be between  $\Delta \ceil{\mu}^\nu$ and $(1-\Delta)^{-1} \ceil{\mu}^\nu$ for sufficiently large $\nu$.
\end{proof} 

\subsection{Proof of Result~\ref{res:1}}
\label{app:B}
To show Result~\ref{res:1}(i) and (ii), consider the derivative with respect to $y$ of $\log(\mu^y/y!) = y \log \mu - \log \Gamma(y+1)$, which is given by $\log \mu - \psi(y+1)$ where $\psi(\cdot) = \Gamma'(\cdot)/\Gamma(\cdot)$ is the digamma function. By known inequalities, 
\begin{eqnarray*}
\psi(y+1) \le  \log(y+1) - \frac{1}{2(y+1)}  \quad \text{ and } \quad \psi(y+1) \ge \log(y + 1/2)  \quad \text{ for all $y \ge 0$} \ ,
\end{eqnarray*}
the derivative is therefore positive for $0 \le y \le\mu-1$ and negative for $y\ge \mu$, which establishes these two results.

To show Result~\ref{res:1}(iii), set $y = 2 \mu^2$ and note that
\begin{eqnarray*}
\log(y!) &>& \underbrace{\log(\mu) + \log(\mu+1) + \ldots + \log(\mu^2 -1)}_{\mu^2-\mu \text{ terms}} +\underbrace{\log(\mu^2) + \log(\mu^2 +1) + \ldots + \log(2\mu^2)}_{\mu^2 + 1 \text{ terms}} \\
& > & (\mu^2 - \mu) \log(\mu) + (\mu^2+1) \log(\mu^2)  \\
&=& (3 \mu^2 - \mu +2) \log(\mu)
\end{eqnarray*}
Hence for  $y = 2 \mu^2$ we have
\begin{eqnarray*}
y \log (\mu) - \log(y!) &<& 2 \mu^2 \log(\mu) - (3 \mu^2 - \mu + 2)   \log(\mu) \\
&=& - (\mu^2 - \mu + 2) \log(\mu) \\
& \le & 0 \ , \quad \text{ for any integer } \mu \ge 1.
\end{eqnarray*}
Conclude that $\mu^y/y! < 1$ strictly for all $ y \ge 2 \mu^2$ from the monotonicity  property in Result~\ref{res:1}(ii). The proof of Result~\ref{res:2} is analogous and therefore omittted.

\section{Proof of Proposition~\ref{eq:prop1}}
\label{app:C}
When $\mu$ is integer, applying Lemma~\ref{lem:1} to the ratio of successive probabilities~(\ref{eq:successive}) gives
$$
\frac{P(Y = \mu - 1)}{P(Y = \mu)} = \frac{\mu^\nu}{\lambda} \to 0 \quad \text{ and} \quad \frac{P(Y = \mu)}{P(Y = \mu+1)} = \frac{(\mu+1)^\nu}{\lambda}\to \infty \quad \text{ as } \nu \to \infty \ .
$$
Hence, the probability at $\mu$ dominates all probabilities to the left and right of $\mu$. Because the total probability must sum to 1, it must be that $P(Y = \mu) \to 1$ and all other probabilities limit to 0.

Similarly, when $\mu$ is non-integer, applying Lemma~\ref{lem:1} or Lemma~\ref{lem:2} to the ratio of successive probabilities~(\ref{eq:successive}) gives
$$
\frac{P(Y = \floor{\mu} - 1)}{P(Y = \floor{\mu})} \to 0 \quad \text{ and} \quad \frac{P(Y = \ceil{\mu})}{P(Y = \ceil{\mu}+1)} \to \infty \quad \text{ as } \nu \to \infty \ ,
$$
implying that the probability at $\floor{\mu}$ dominates probabilities to the left of it and the probability at $\ceil{\mu}$ dominates probabilities to the right of it. Hence, the limiting distribution can have, at most, two non-zero probabilities at the values  $\floor{\mu}$ and $\ceil{\mu}$. Finally, applying Lemma~\ref{lem:2} to the ratio of successive probabilities at $\floor{\mu}$ and $\ceil{\mu}$ gives
$$
1-\Delta < \frac{P(Y = \floor{\mu})}{P(Y = \ceil{\mu})} < \frac{1}{\Delta} \ , \quad \text{ for sufficiently large } \nu \ ,
$$
which implies that the ratio $P(Y = \floor{\mu})/P(Y = \ceil{\mu})$ must converge to some constant as $\nu \to \infty$. Thus, it must be that $P(Y = \floor{\mu}) \to 1-\Delta$ and $P(Y = \ceil{\mu}) \to \Delta$ for the mean to remain fixed at $\mu$, and so the ratio $P(Y = \floor{\mu})/P(Y = \ceil{\mu})$ converges to $(1-\Delta)/\Delta$.

\end{document}